\documentclass[11pt]{article}

\usepackage[margin=1in]{geometry}
\usepackage{xcolor}
\usepackage{amsmath,amssymb,amsthm}
\usepackage{graphicx}
\usepackage{subcaption}
\usepackage{xr}
\usepackage{placeins}
\usepackage{url}
\usepackage[hidelinks]{hyperref}

\newtheorem{theorem}{Theorem}
\newtheorem{lemma}{Lemma}
\newtheorem{corollary}{Corollary}

\theoremstyle{definition}

\theoremstyle{remark}

\title{\textbf{Asymptotic Null Distributions of Moran's $I$ and Assortativity in Large Networks}\thanks{This is a preprint submitted for review at the Journal of Complex Networks and has not yet been peer-reviewed in its current form.}}

\author{
Karin Ait Braham$^{\dagger}$ \qquad
Louis-Paul Rivest \qquad
Thierry Duchesne\\[0.5em]
\small Department of Mathematics and Statistics, Université Laval\\
\small $^{\dagger}$Corresponding author:
\texttt{karin.ait-braham.1@ulaval.ca}
}

\date{}

\begin{document}

\maketitle

\begin{abstract}
This study investigates the asymptotic behavior of two dependence measures
defined on networks, Moran's $I$ statistic and Newman's assortativity,
{under the null hypothesis that a Gaussian node attribute $Y$
is independent of the network structure}.
We demonstrate that the structure of the network directly affects the
convergence rate to normality of these measures as the size of the network
increases. We further establish that, in some instances, the mean values of
these dependence measures under the null hypothesis remain non-negligible
asymptotically and must therefore be explicitly accounted for when calculating
the test statistics. Applications to a variety of simulated and real networks
also reveal that the normal approximation performs well only when the network
is not strongly heterogeneous.
{Network topology determines both the convergence rate to
normality and whether the limiting distribution is Gaussian. In dense networks
whose degree heterogeneity does not vanish, we further show that assortativity
can fail to be a valid test statistic even though Moran's $I$ remains well
behaved, whereas a dominating node invalidates both.}
\end{abstract}


\section{Introduction}
{In network data analysis, one question remains unanswered: how can we distinguish what truly arises from the structure of the network from what is merely random fluctuation?}	
Network science has emerged as a major interdisciplinary field drawing on sociology, computer science, statistical physics, and economics. From the early sociograms of \cite{moreno1934shall} to the foundational probabilistic modeling work of \cite{erdHos1960evolution}, numerous concepts have been developed to better understand network structure and dynamics. The centrality measures introduced by \cite{freeman1979centrality}, the work of \cite{milgram1967six} on social proximity, and more recent contributions by \cite{newman2002,newman2003,newman2012communities} on complex and bipartite networks have significantly enriched the field. The rise of online networks \cite{mislove2007measurement} and the development of modern statistical methods for covariance estimation on graphs (\cite{chen2018estimating}; \cite{lan2018covariance}) further illustrate the growing importance of statistical tools in structural network analysis.\\

	Among the many metrics introduced to quantify dependence in a network, assortativity (\cite{newman2002,newman2003}) plays a central role. As a second-order statistic, it evaluates the tendency of connected nodes to share similar values of the attribute studied $Y$, offering a more nuanced interpretation of local interactions. {Despite its widespread use, the asymptotic properties of assortativity measures remain less understood than those of Moran's $I$, particularly in the context of large heterogeneous networks.}\\
    
	This issue naturally parallels questions encountered in spatial statistics, where one seeks to detect dependence between observations according to their geographic proximity. Moran’s I coefficient $I$ (\cite{moran1950notes}) is the classical tool used to measure spatial correlation. This statistic was generalized by \cite{cliff_ord1973,cliff_ord1981}, who formalized spatial interaction models and established asymptotic properties. \cite{sen1976large} proved the asymptotic normality of Moran's $I$ statistic for non-normal data under finite fourth moments.
	Spatial econometrics later deepened this analysis by incorporating neighborhood effects directly within explanatory variables and disturbance terms \cite{anselin1988spatial}. In such models, Moran's $I$ statistic  remains a standard diagnostic tool, whether or not an autoregressive component is present (\cite{lin2009properties} and \cite{kelejian_prucha2001}). All these contributions establish asymptotic normality with the classical $\sqrt{n}$ convergence rate.\\
    
    {
    The link between spatial autocorrelation  and network dependence has been increasingly recognized in recent years. For instance, \cite{lee2020testing} emphasized that classical tools for detecting spatial autocorrelation can naturally be adapted to network data and proposed testing procedures applicable in both settings. These developments suggest that statistics such as Moran's $I$ and Newman's assortativity may be studied within a common inferential framework. 
   Network topology, as measured by eigenvalues of the adjacency matrix,  may substantially affect the asymptotic behavior of these statistics. In particular, when  individual eigenvalues become negligible relative to the overall spectral mass of the network, standardized versions of Moran's $I$ and assortativity are asymptotically normal. Conversely, for highly structured networks, such as star networks or community networks with a fixed number of large communities, a finite number of eigenvalues remain asymptotically non-negligible. In such settings, the limiting distributions are no longer Gaussian. Examples illustrating both regimes are discussed in Section~5.}
   \\
   
   { A Gaussian distribution for $Y$ is adopted throughout this paper because the theoretical analysis relies on spectral representations of quadratic forms in Gaussian random vectors. Although this framework covers many applications involving continuous and approximately normal responses that are approximately symmetric and light-tailed, possibly after a transformation, extending the results to non-Gaussian settings would be of considerable interest. One could possibly build on  
  \cite{sen1976large} who shows that normality can be replaced by the existence of a fourth moment, when the network has uniformly bounded degrees. 
  Throughout this paper, the adjacency matrix is treated as fixed.
  It could be the realization of an underlying stochastic mechanism, such as the  $\beta$-model \cite{yan2013central}, the $p_0$ model \cite{yan2021directed,yan2016asymptotics}, and more recent covariate-adjusted network models \cite{qu2026inference,yan2019statistical}. The proposed test statistic is conditional on the observed network: the null distribution of the response vector does not depend on the network structure, so that any apparent association between connected nodes is attributable to random variations in $Y$.}\\
  
      {This work investigates the asymptotic distribution of the assortativity coefficient and of Moran's statistic when the proximity between units is determined by the adjacency matrix of a fixed network. Our main results characterize network structures for which asymptotic normality holds  and provide examples where the limiting distribution is chi-squared.}
    The analysis is based on Lindeberg’s central limit theorem \cite[Sec.~27, p.~310, Th.~27.2]{billingsley2017}. It provides a solid theoretical framework for comparing the performance of these statistics and understanding in which structural settings they can detect an association between the network and the variability of the random variable~$Y$.
    {Testing the independence of a network and a node attribute is the primary question
in several applications, among them the insurance networks of Section~5.3.2, where a continuous
measure of client involvement is available and the question is whether it is associated to the clients networks.}\\
    
    \section{Presentation of the two statistics}	
	In this section, we present two statistics used to test whether the structure of a network has an effect on the values of a normal random variable $Y$: the assortativity coefficient $r_n$ and Moran’s coefficient $I_n$. We then derive their expressions for several classical families of networks in order to highlight the influence of network topology.
	\subsection{Notation}	
	We consider a simple   unweighted network with $n$ nodes.  Its adjacency matrix  $A_n=\{a_{ij}: \ 1\le i,j\le n\}$ is a symmetric matrix satisfying $a_{ii}=0, \ i=1,\ldots,n$, and $a_{ij}\in \{0,1\}$ for all $i\ne j$, with $a_{ij}=1$ if, and only if, nodes $i$ and $j$ are connected. The following notation is used:
	\begin{itemize}
		\item 
        $\mathbf{I}_n$ is the $n\times n$ identity matrix;
		\item 
        $\mathbf{1}_n$ is the $n$-dimensional vector of ones;
		\item 
        $\mathbf{d}_n = A_n \mathbf{1}_n = (d_{n1},\ldots,d_{nn})^{\top}$ is the vector of degrees,
		with $d_{ni}$ the degree of node $i$;
		\item 
        $D_n=\operatorname{diag}(d_{n1},\ldots,d_{nn})$: the diagonal matrix of degrees;
		\item 
        $m_n$: the total number of edges, so that $\sum_{i=1}^n d_{ni} = 2m_n$;
        \item
 
        $\bar d_n = \dfrac{2m_n}{n}$ is the average degree;
		\item 
        $s_{nd}^2 = \dfrac{1}{n}\sum_{i=1}^n (d_{ni}-\bar{d}_n)^2$ is the variance of the degrees;
		\item 
        $CV_{nd} = \dfrac{s_{nd}}{\bar d_n}$ is the coefficient of variation of the degrees;
		\item $\mathbf{Y}_n=(Y_1,\ldots,Y_n)^{\top}$: the observations of the variable $Y$;
		\item 
        $\bar Y_n ={n}^{-1}\sum_{i=1}^n Y_i$ and $\bar Y_d
	= \sum_{i=1}^n d_{ni} Y_i/(\sum_{i=1}^n d_{ni})$ are  the unweighted and degree weighted empirical means;
		\item 
        $M_n = \mathbf{I}_n - \frac{1}{n}\mathbf{1}_n\mathbf{1}_n^{\top}$: the centering matrix.
	\end{itemize}
A network is said to be regular if all its degrees are equal: $d_{ni}=\bar d_n$ for $i=1,\ldots,n$.
\\
 
{The objective is to assess whether the network structure is associated with the
variability of the node attributes $\mathbf Y_n$. The null
hypothesis is that $Y_1,\ldots,Y_n$ are independent and identically distributed as
$N(\mu,\sigma^2)$, so that their joint distribution, a $N_n(\mu\mathbf{1}_n, \sigma^2 \mathbf{I}_n)$, does not depend on $A_n$. Under the
alternative, the $Y_i$ remain marginally $N(\mu,\sigma^2)$ but are no longer
independent, their dependence being induced by the network structure encoded in
$A_n$. To make this precise, define
\begin{equation}\label{eq:Qn-def}
Q_n=(\mathbf Y_n-\mu\mathbf 1_n)^\top A_n(\mathbf Y_n-\mu\mathbf 1_n).
\end{equation}
Since $a_{nii}=0$ for all $i$,
$$E(Q_n)=2\sum_{i<j}a_{nij}\operatorname{Cov}(Y_i,Y_j),$$
which is zero under $H_0$ (independence) and generally nonzero under $H_1$
(network-induced dependence). The two hypotheses of interest are
\begin{equation}\label{eq:hypotheses}
H_0:E(Q_n)=0 \qquad\text{against}\qquad
H_1:E(Q_n)>0 \ \text{ or }\ H_1:E(Q_n)<0.
\end{equation}
Two test statistics for $H_0$, the assortativity and Moran's $I$, are now introduced. }



	\subsection{Assortativity Coefficient $r_n$}
	The assortativity coefficient, introduced by \cite{newman2002,newman2003}, measures the tendency of connected nodes to share similar values of the variable studied $Y$. In our setting, it
	is defined as
	\begin{equation}\label{eq03}
		r_n
		= \frac{\displaystyle \sum_{i=1}^{n}\sum_{j=1}^{n} a_{ij}\,(Y_i-\bar Y_d)(Y_j-\bar Y_d)}
		{\displaystyle \sum_{i=1}^{n} d_{ni}\,(Y_i-\bar Y_d)^2},
	\end{equation}
	and $\bar Y_d$ is the degree weighted mean.  This statistic is normalized and takes values in $[-1,1]$.	
	\subsection{Moran's Coefficient $I_n$}	
	Moran's coefficient \cite{moran1950notes} is a classical tool for quantifying a spatial correlation among observations associated with connected units. When the proximity between units is defined through the adjacency matrix $A_n$, Moran's coefficient is given by
	\begin{equation}\label{eq004}
		I_n
		= \frac{n}{2m_n}\,
		\frac{\displaystyle \sum_{i=1}^{n}\sum_{\substack{j=1\\ j\neq i}}^{n}
			a_{ij}\,(Y_i-\bar Y_n)(Y_j-\bar Y_n)}
		{\displaystyle \sum_{i=1}^{n}(Y_i-\bar Y_n)^2}.
	\end{equation}
	Unlike $r_n$, the statistic $I_n$ can be larger than 1. In spatial statistics, the definition of $I_n$ involves a weight matrix 
whose entries are not necessarily 0 or 1. {The null hypothesis $H_0:\mathbb E(Q_n)=0$ is rejected when $r_n$ or $I_n$ is either large positive or large negative.}\\	

For subsequent asymptotic developments, we rewrite $r_n$ and $I_n$ as ratios of quadratic forms,
 \begin{equation}\label{eq:ratio}
 r_n = \frac{S_{A,n}}{S_{D,n}} \mbox{~~and} ~~I_n = \frac{S_{AM,n}}{S_{M,n}},
\end{equation}
{
where the numerator and the denominator for $r_n$ and $I_n$ are given by
\begin{align*}
	S_{A,n}
	&= \frac{1}{2m_n}\,\mathbf{Y}_{n}^{\top}\!\left( A_{n}
	-\frac{1}{2m_{n}}\,\mathbf{d}_{n}\mathbf{d}_{n}^{\top} \right)\!\mathbf{Y}_{n},& 
    S_{D,n}
	& = \frac{1}{2m_n}\,\mathbf{Y}_{n}^{\top}\!\left( D_{n}
	-\frac{1}{2m_{n}}\,\mathbf{d}_{n}\mathbf{d}_{n}^{\top} \right)\!\mathbf{Y}_{n} \\
	S_{AM,n}
	&= \frac{1}{2m_n}\mathbf Y_n^\top M_n A_n M_n\mathbf Y_n,& 
    S_{M,n}
	& = \frac{1}{n}\,\mathbf{Y}_{n}^{\top} M_n\mathbf{Y}_{n}.
	\end{align*}
These quadratic forms are obtained by noting that the numerator of $r_n$ is equal to   
\[
S_{AM,n}=\sum_{i,j=1}^{n}a_{ij}(Y_i-\bar Y_d)(Y_j-\bar Y_d)/(2m_n)
=
\{\sum_{i,j=1}^{n}a_{ij}Y_iY_j\}/(2m_n)-
\{\sum_{i=1}^{n}d_{in}Y_i\}^2/(2m_n)^2,
\]
while that for $I_n$ is 
\[
S_{AM,n}=\sum_{i,j=1}^{n}a_{ij}(Y_i-\bar Y_n)(Y_j-\bar Y_n)/(2m_n)
=
(M_n\mathbf Y_n)^\top A_n(M_n\mathbf Y_n)/(2m_n).
\]
The quadratic forms for the denominators are obtained in a similar way.}
Observe that for a regular network $r_n=I_n$.	
\section{Some closed form expressions for $r_n$ and $I_n$}\label{sec:example}
This section presents some simple networks for which $I_n$ and $r_n$ have simple explicit expressions that will serve as test cases for our asymptotic results.
	\subsection{Star Network}\label{sec:star}
	\begin{figure}[!h]
	    \centering\includegraphics{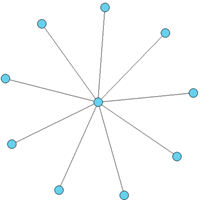}
	    \caption{Star network with $n=10$ vertices.}
	    \label{fig:star}
	\end{figure}	
	Consider a star network with $n=10$ nodes given in Figure \ref{fig:star}.  It consists of a central node, numbered 1, connected to the $n-1$ peripheral nodes that are numbered 2 to $n$. Moran's statistic is easily shown to be equal to
	\begin{equation}\label{eq:In-star}
		I_n
		=
		-\frac{n}{n-1}
		\frac{(Y_1-\bar{Y}_n)^{2}}
		{\sum_{i=1}^n (Y_i-\bar{Y}_n)^{2}},
	\end{equation}
	 while 
     {
     \begin{equation}\label{eq:rn-star}
         r_n = -\frac{n^{2}(Y_1-\bar{Y}_n)^{2}}
         {2(n-1)\sum_{i=1}^{n}(Y_i - \bar{Y}_n)^2 
         + n(n-2)(Y_1 - \bar{Y}_n)^2}.
     \end{equation}
     }
     Observe that \eqref{eq:In-star} is proportional to the square of the z-score used  to test whether $Y_1$ is an outlier in a normal sample, {see \cite[sect.~6.3]{barnett1994outliers} for further discussion.  }   
	\subsection{Complete Network} \label{sec:comp}
	\begin{figure}[!h]
	    \centering\includegraphics{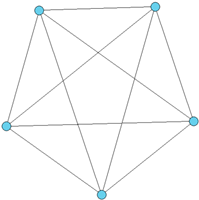}
	    \caption{Complete network with $n=5$ vertices.}
	    \label{fig:comp}
	\end{figure}	
	Figure \ref{fig:comp} shows a complete network on $n$ nodes. This is a regular network where each node
	has degree $n-1$ and $A_n = \mathbf{1}_n\mathbf{1}_n^{\top} - I_n$ and $2m_n = n(n-1)$.
	It follows immediately that $I_n = r_n = -\,\frac{1}{n-1}$.\\
	Thus, in a  fully connected network, both statistics are equal to $-\frac{1}{n-1}$ with probability $1$.	
	\subsection{Path Network}\label{sec:pat}
	\begin{figure}[!h]
	    \centering\includegraphics{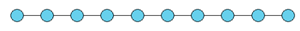}
	    \caption{Path network with $n=10$ vertices.}
	    \label{fig:path}
	\end{figure}	
	Consider the adjacency matrix $A_n$ corresponding to a path with edges $(1,2),\dots,(n-1,n)$, as shown in Figure~ \ref{fig:path}.  In this case,
	\[
	\sum_{\substack{i,j=1 \\ i\neq j}}^{n}
	a_{ij}\,(Y_i-\bar Y_n)(Y_j-\bar Y_n)
	=
	2\sum_{i=1}^{n-1}(Y_i- \bar{Y}_n)(Y_{i+1}-\bar{Y}_n).
	\]
	Thus,
	\[
	I_n =
	\frac{n}{n-1}\,
	\frac{\displaystyle \sum_{i=1}^{n-1}
		(Y_i- \bar{Y}_n)(Y_{i+1}-\bar{Y}_n)}
	{\displaystyle \sum_{i=1}^{n} (Y_i-\bar{Y}_n)^2 },
	\]
	is proportional to the empirical auto-correlation of order 1 for the time series $\{Y_i:\ i=1,\ldots,n\}$.
	\subsection{Community Network}\label{sec:com}
    Consider a network composed of $k$ disjoint communities, each of size $\ell_n=n/k$, presented in Figure \ref{fig:clus}. Each community
	 is a complete network of size $\ell_n$ as defined in Section \ref{sec:comp}. This is a regular network where  $r_n=I_n$ can be rewritten using an ANOVA-type decomposition.
	\begin{figure}[!h]
	    \centering\includegraphics{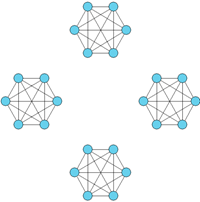}
	    \caption{Four-cluster network with four disjoint communities of size 6 $(n = 24)$.}
	    \label{fig:clus}
	\end{figure}	
    It is convenient to let the index $i=1,\ldots,k$ denote a community and $j=1,\ldots,\ell_n$ denote a unit within a community. { The observations $Y_1,\ldots,Y_n$ are re-indexed accordingly, with $Y_{ij}$ denoting the response observed for unit $j$ in community $i$. The corresponding community and overall means are
\[
\bar Y_i=\frac{1}{\ell_n}\sum_{j=1}^{\ell_n}Y_{ij},
\qquad
\bar Y_n=\frac{1}{n}\sum_{i=1}^{k}\sum_{j=1}^{\ell_n}Y_{ij}.
\]}
	The standard sums of squares of the one way ANOVA are
	\[
	\mathrm{SST}
	= \sum_{i,j} (Y_{ij} - \bar Y_n)^2,
	\qquad
	\mathrm{SSB}
	= \sum_{i=1}^k \ell_n (\bar Y_i - \bar Y_n)^2,
	\qquad
	\mathrm{SSW}
	= \sum_{i,j} (Y_{ij} - {\bar{Y}_i})^2,
	\]
	with the identity $\mathrm{SST} = \mathrm{SSB} + \mathrm{SSW}$.
	For this network
    \begin{eqnarray*}
(\mathbf{Y}_n-\bar Y_n\mathbf{1}_n)^\top A_n (\mathbf{Y}_n-\bar Y_n \mathbf{1}_n)&=&
	\sum_{i}
	\left[
	\left\{ \sum_{j}(Y_{ij}-\bar Y_n) \right\}^2
	-
	\sum_{j}(Y_{ij}-\bar Y_n)^2
	\right] \\
    &=&
	\ell_n\ \mathrm{SSB}
	- \mathrm{SST}.
    \end{eqnarray*}	
		Therefore,
\[
	I_n
	= \frac{n}{k\ell_n(\ell_n-1)}\frac{(\mathbf{Y}_n-\bar Y_n\mathbf{1}_n)^\top A_n (\mathbf{Y}_n-\bar Y_n\mathbf{1}_n)}{\mathrm{SST}}= \frac{\ell_n R^2 - 1} {\ell_n-1},
	\]
where $R^2$ is  the ANOVA coefficient of determination
	$R^2 = \mathrm{SSB}/\mathrm{SST}$. Thus, for this network, the size of $I_n=r_n$ depends on the differences between the means of the community of $Y$. Indeed, one can show that when $n$ is large, $I_n$ is equal to the adjusted $R^2$ for the one way ANOVA of $Y$ as a function of the factor community. This helps in the interpretation of $I_n$. 	
	\section{Asymptotic Normality of the Statistics $r_n$ and $I_n$}\label{sec:normalite}	
This section establishes the asymptotic normality of $r_n$ and $I_n$  by finding the limits of
	 the quadratic forms $S_{A,n}$, $S_{D,n}$, $S_{AM,n}$, and $S_{M,n}$ introduced in equation \eqref{eq:ratio}.
 We prove that the normalized denominators $S_{D,n}/\sigma^2$ and $S_{M,n}/\sigma^2$ converge to 1 in probability, while we use a central limit theorem to prove the convergence in distribution of the numerator $S_{A,n}/\sigma^2$ and $S_{M,n}/\sigma^2$.\\
The convergence in probability of the denominators is established in the next lemma, which is proved using a Weak Law of Large Number.
	\begin{lemma}[Convergence in probability of the denominators]\label{lem:denom}
		Let $\mathbf{Y}_{n}$ be a random vector with a $N_n(\mu\mathbf{1}_n,\sigma^{2}\mathbf{I}_n)$ distribution. The normalized denominator of $I_n$, $S_{M,n}/\sigma^2$, converges to 1 in probability, while a sufficient condition for the convergence of $S_{D,n}/\sigma^2$
to 1 in probability is that
 \begin{equation}\label{eq:cv-cond}
			\frac{CV_{nd}^{2}}{n}\xrightarrow[n\to\infty]{}0.
		\end{equation}	    
	\end{lemma}
	Condition \eqref{eq:cv-cond} of "no dominating node" is needed to establish that $S_{D,n}/\sigma^2$ converges to 1 in probability. 
\medskip
	To demonstrate the asymptotic normality of the numerators $S_{A,n}$ and $S_{AM,n}$ we apply Lindeberg’s central limit theorem that is now stated.
    
	\begin{theorem}[Lindeberg]\label{thm3}
		Let $(\Omega_{n}, \mathcal{F}_{n},\mathbb{P})$ be a probability space for each $n$, and let
		$X_{n1},\ldots, X_{nN_{n}}$ be independent random variables on this space, with
		$\mathbb{E}(X_{ni})=\mu_{ni}$ and $\mathrm{Var}(X_{ni})=\sigma_{ni}^{2}\in(0,\infty)$. Define
		$s_{n}^{2}=\sum_{i=1}^{N_n}\sigma_{ni}^{2}$. If, for every $\epsilon>0$,
		\begin{equation}\label{condLind}
			\lim_{n\to\infty}
			\frac{1}{s_{n}^{2}}\sum_{i=1}^{N_n}
			\mathbb{E}\left[(X_{ni}-\mu_{ni})^{2}
			\mathbf{1}_{\{|X_{ni}-\mu_{ni}|>\epsilon s_{n}\}}
			\right] =0,
		\end{equation}
		then
		\[
		\frac{\sum_{i=1}^{N_n}(X_{ni}-\mu_{ni})}{s_{n}}
		\xrightarrow[n\to\infty]{\mathcal{L}} \mathcal{N}(0,1).
		\]	    
	\end{theorem}
The next theorem determines the asymptotic distributions of $r_n$ and $I_n$. Its proof, together with that of Corollary \ref{cor:sparse_regime-asym} is given in Section \ref{sec:proof}.
	\begin{theorem}[Asymptotic normality of $r_n$ and $I_n$]\label{thm1}
		Let $\mathbf{Y}_{n}$ be a vector of $n$ i.i.d.\ variables with distribution $N(\mu,\sigma^{2})$ and define the asymptotic means of $r_n$
and $I_n$ by
		\begin{align*}
			 \mu(r_n)&= -\frac{1}{n}\,\bigl(CV_{nd}^{2}+1\bigr), & \mu(I_n)&= -\frac{1}{n},
\end{align*}
and their asymptotic variances by
\begin{align*}
			 \sigma^2(r_n)&= \frac{2}{n^{2}}
			\left[\frac{n}{\bar d_{n}}+
			CV_{nd}^{4}-1-2CV_{nd}^{2}I_{nd}-2CV_{nd}^{2}
			\right],\\
		 \sigma^2(I_n)&=\frac{2}{n^{2}}
			\left[\frac{n}{\bar d_{n}}
			-1-2CV_{nd}^{2}
			\right],
		\end{align*}
where $I_{nd}$ is Moran's $I$ statistic calculated on the vector of degrees, $\mathbf{d}_n$.\\
Assume that {
$n\sigma(I_n)$ is bounded away from 0 as $n$ becomes large and
}
		\begin{equation}\label{eq:spectral-cond}
			\lambda_{n,\max}^*=\frac{\lambda_{n,\max}}{\sqrt{\sum_{i=1}^{n}\lambda_{ni}^{2}}}
			\xrightarrow[n\to\infty]{}0,
		\end{equation}
		where $\lambda_{ni}$, $i=1,\dots,n$, are the eigenvalues of the matrix $B_n=M_{n}A_{n}M_{n}$ and
		$\lambda_{n,\max}=\max_{1\le i\le n}|\lambda_{ni}|$, then
		\begin{equation}\label{eq:clt-rn}
			\frac{I_{n}-\mu(I_{n})}{\sigma(I_{n})}
			\xrightarrow[n\to\infty]{\mathcal{L}} \mathcal{N}(0,1).
		\end{equation}
 Result \eqref{eq:clt-rn} also holds for $r_n$ if we replace $B_n$ by $A_{n} - \,\mathbf{d}_{n}\mathbf{d}_{n}^{\top}/(2m_{n})$ and $\mu(I_n),\sigma(I_n)$ by $\mu(r_n),\sigma(r_n)$ provided that  $n\sigma(r_n)$ is bounded away from 0 as $n$ becomes {large}.
	\end{theorem}
{ The assumptions on $\sigma(r_n)$ and $\sigma(I_n)$ mean that these two expressions cannot go to 0 at a rate faster that $1/n$.} The two statistics $r_n$ and $I_n$ only differ in the way in which the $Y$ values are centered, either with the mean $\bar Y_n$ or with the degree weighted mean $\bar Y_d$.  They have the same limiting distribution when the network is not dense, that is when
$\bar d_n \mbox{~is~} o(n)$. 
This is stated in the next corollary whose proof is given in Section \ref{sec:cor1}.
\begin{corollary}\label{cor:sparse_regime-asym}
		If
		\[
		\frac{\bar d_n} n   \xrightarrow[n\to\infty]{} 0
		\qquad\text{,}\qquad
		\frac{d_{n,\max}}{\sqrt{n\,\bar d_n}} \xrightarrow[n\to\infty]{} 0
        \qquad\text{and}\qquad CV_{nd}\mbox{~is~} O(1),
		\]
		where $d_{n,\max}$ is the maximum entry in $\mathbf{d}_n$, then the assumptions of Theorem~\ref{thm1}, \eqref{eq:cv-cond}--\eqref{eq:spectral-cond},
are satisfied, and $r_n$ and $I_n$ share the same centered Gaussian limit:
		\[
		\sqrt{\frac{n\,\bar d_n}{2}}\,r_n
		\;\xrightarrow[n\to\infty]{\mathcal L}\;
		\mathcal{N}(0,1), \qquad
\sqrt{\frac{n\,\bar d_n}{2}}\,I_n
		\;\xrightarrow[n\to\infty]{\mathcal L}\;
		\mathcal{N}(0,1).
		\]
	\end{corollary}	
	In  Corollary \ref{cor:sparse_regime-asym}, the rate of convergence to normality is $\sqrt{n}$ when the average degree $\bar d_n$
	converges to a positive constant. This is one of the hypotheses made on page 49 of \cite{cliff_ord1981} and in \cite{sen1976large} to establish the asymptotic normality of $I_n$. Our result is more general, as it relates the convergence rate to network’s characteristics.
     \\

We now consider dense networks whose density $\bar d_n/n$ converges to a positive value $\alpha$ as $n$ goes to $\infty$.  The next corollary is concerned with regular networks in which all degrees are equal.	
	\begin{corollary}[Dense regular network]\label{cor:random_regular_dense-asym}
		Suppose that the adjacency matrix $A_n$ is selected uniformly among all possible $n \times n$ regular adjacency matrices with degree $d_n$ where  $d_n/n$ converges to $ \alpha$ as $n$ goes to $\infty$ where 
		$\alpha \in (0,\tfrac12]$. Then  there exists $n_0=n_0(\alpha)$ such
		that for all $n\ge n_0$, the random adjacency matrix $A_n$ satisfies assumption \eqref{eq:spectral-cond} of Theorem~\ref{thm1} with probability at
		least $1-1/n$. For such matrices $A_n$, the limiting distribution of $r_n=I_n$ is given by
		\[
		\frac{n\,I_n}{\sqrt{\frac{2(1-\alpha)}{\alpha}}}
		\;+\;
		\sqrt{\frac{\alpha}{2(1-\alpha)}}
		\xrightarrow[n\to\infty]{\mathcal L}
		\mathcal{N}(0,1).
		\]
	\end{corollary}	
	\begin{proof}
		Let $\lambda_1(\cdot)\le ...\le \lambda_n(\cdot)_n$ denote the ordered eigenvalues of a symmetric $n\times n$ matrix. If the network is regular then the largest eigenvalue of $A_n$ is $\lambda_n(A_n)=d_n$ with eigenvector $\mathbf{1}_n$. Now in Theorem~\ref{thm1} the eigenvalues of  $B_n=A_n-d_n\mathbf{1}_n\mathbf{1}_n^\top/n$ are equal to 0, $\lambda_1(A_n)$,..., $\lambda_{n-1}(A_n)$.   Thus, in condition \eqref{eq:spectral-cond}, 
        $\lambda_{n,\max}
		= \max(|\lambda_{1}(A_n)|,|\lambda_{n-1}(A_n)|)$.
       According to Theorem~A of \cite{tikhomirov2019spectral} 
       \begin{equation}\label{eq:tik}
           \lambda_{n,\max} \le L\sqrt{d_n}\approx L\sqrt{\alpha n},
       \end{equation} 
       with a large probability for $n$ large enough.
		Furthermore, $\sum_{i=1}^{n}\lambda_{ni}^{2}$ is equal to the trace of $A_n^2$, $nd_n$, minus the square of its largest eigenvalue, $d_n^2$.
		So $\sqrt{\sum_i \lambda_{ni}^{2}} \asymp n\sqrt{\alpha(1-\alpha)}$. Together with inequality \eqref{eq:tik}, this proves that \eqref{eq:spectral-cond} is verified and the result of the corollary is true with a high probability.
        	\end{proof}

In Corollary \ref{cor:random_regular_dense-asym} the rate of convergence to normality is $n$.  This is the fastest possible rate in Theorem \ref{thm1}.	
	\section{Some applications of Theorem \ref{thm1}}	
    
    This section applies Theorem \ref{thm1} to some specific networks. {We first ask, for the examples of Section~\ref{sec:example}, whether the assumptions of Theorem~\ref{thm1} are met.} {Then we illustrate the convergence to normality as $n$ becomes large and ascertain whether the hypotheses of Theorem~\ref{thm1} are met, both for simulated and for real networks.}

\subsection{The networks of Section \ref{sec:example}}	
	For the star network of Section \ref{sec:star}, Theorem \ref{thm1} does not hold because the distribution of the degrees $d_{n,i}$ is very skewed, $CV^2_{nd}$ being $O(n)$.
{
The two statistics then behave differently. Since $\bar Y_n$ and $\sum_{i=1}^n(Y_i-\bar Y_n)^2/n$ converges in probability to $\mu$ and $\sigma^2$ respectively, \eqref{eq:In-star} gives
\[
-nI_n \xrightarrow[n\to\infty]{\mathcal{L}} \frac{(Y_1-\mu)^2}{\sigma^2}\sim\chi^2_1,
\]
so $I_n$  converges, at a rate of $n$, to a non normal distribution. The limiting distribution of $r_n$ is different: dividing the numerator and the denominator of \eqref{eq:rn-star} by $n^2$ shows that
\[
r_n \xrightarrow[n\to\infty]{\mathcal{L}} -\frac{(Y_1-\mu)^2}{2\sigma^2+(Y_1-\mu)^2},
\]
a nondegenerate random limit with values in $(-1,0]$. 
} \\

    For the path network of Section \ref{sec:pat}, $\bar d_n \approx 2$. Corollary \ref{cor:sparse_regime-asym} applies and gives the classical result that the limiting distribution of $\sqrt{n}$ times the first order autocorrelation in a normal sample is $\mathcal{N}(0,1)$.\\
    
	The community network of Section \ref{sec:com} is interesting because whether Theorem \ref{thm1} applies or not depends on the size $\ell_n$ of a community. Since the network is regular, its degrees are all equal to $\ell_n-1$ and the multiplicity of $\ell_n-1$, the largest eigenvalue of $A_n$, is $k$, the number of communities.   If $\ell_n$ is $o(n)$, then Corollary~\ref{cor:sparse_regime-asym} applies and the limiting distribution is normal.  If $\ell_n$ is $O(n)$,  Corollary \ref{cor:random_regular_dense-asym} does not apply because the second largest eigenvalue of $A_n$ is $\ell_n-1$ and condition \eqref{eq:spectral-cond} fails.  The limiting distribution of $I_n$ and $r_n$ is then $\chi^2_{k-1}$. { This follows from the classical theory of quadratic forms in Gaussian random vectors (see, e.g.,
\cite[Sec.~4.4, pp.~156--163]{graybill1976theory}).}	
	\subsection{Simulated networks}\label{sec:simulated}
{To illustrate Theorem~\ref{thm1} and Corollaries~\ref{cor:sparse_regime-asym} and~\ref{cor:random_regular_dense-asym}, we consider three families of random networks that represent distinct asymptotic regimes: }
\begin{itemize}
    \item[{(i)}] An \textbf{Erdős--Rényi random network} with a probability of 0.4 of having an edge between two nodes.
    \item[{(ii)}] A  \textbf{Bounded‑degree network} {with degrees uniformly distributed between 2 and 20.}
    \item[{(iii)}] A  \textbf{Dense heterogeneous network} {with degrees uniformly distributed between 5 and $n-5$}.
\end{itemize}
{The networks in (i), (ii), and (iii) are shown in Figure \ref{fig:all_networks}.  They are simulated for $n=25, 50, 75, 100$.} To simulate a Erdős--Rényi network we used the \texttt{sample\_gnp}
function of the  package \texttt{igraph} \cite{igraph}. 
For (ii) and (iii), the simulation proceeded in two steps. First, the degrees $d_{ni}$ were selected uniformly.   
 We then verified whether the degree sequence was \emph{graphical}, that is, satisfied Erdős-Gallai inequalities \cite[ Theorem 2.2, p. 492]{blitzstein2011sequential}.
Finally we  used the algorithm of \cite{tille2025existence}, to select one network uniformly among all those with the simulated degree sequence.\\
\begin{figure}[!h]
		\centering
    \begin{subfigure}[t]{0.30\textwidth}
        \centering
        \includegraphics[width=\textwidth]{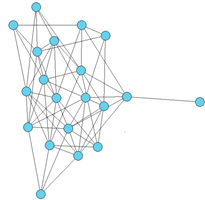}
        \caption{Erdős--Rényi network}
    \end{subfigure}
    \hfill
    \begin{subfigure}[t]{0.30\textwidth}
        \centering
        \includegraphics[width=\textwidth]{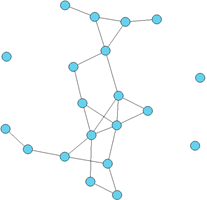}
        \caption{Bounded-degree network}
    \end{subfigure}
    \hfill
    \begin{subfigure}[t]{0.30\textwidth}
        \centering
        \includegraphics[width=\textwidth]{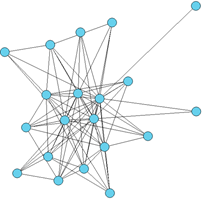}
        \caption{Dense heterogeneous network}
    \end{subfigure}
        \caption{Representation of the three simulated networks when $n=100$. For each one a sub-network obtained by randomly selecting 20 nodes is shown.}
         \label{fig:all_networks}
	\end{figure}
    \FloatBarrier
    
    {For each network and each $n$, we generated $M=10\, 000$ independent replications of $\mathbf{Y}_n$ under $H_0$ and computed, for each replication, the standardized statistics
\[
r_n^\ast=\frac{r_n-\mu(r_n)}{\sigma(r_n)}, \qquad I_n^\ast=\frac{I_n-\mu(I_n)}{\sigma(I_n)},
\]
where $\mu(\cdot)$ and $\sigma(\cdot)$ are the asymptotic means and standard deviations given by Theorem~\ref{thm1} 
The empirical rejection probability at level $\alpha$ is the proportion of the $M$ replications for which $r_n^\ast$ (resp.\ $I_n^\ast$) falls below the lower $\alpha$-quantile or above the upper $(1-\alpha)$-quantile of the standard normal distribution, for $\alpha\in\{0.01,0.05\}$; these are reported separately as lower- and upper-tail probabilities in Table~\ref{tab:erdos_probabilities}. Similar results for the bounded-degree and the heterogeneous networks are reported in (Tables~S1-S2) of the Supplementary Material. Structural indicators for all three networks are also reported in Tables~S3--S5.}\\

\begin{table}[!h] 
    \centering 
    \setlength{\tabcolsep}{10pt} 
    \caption{{Empirical rejection probabilities of the one-sided tests based on the standardized statistics $r_n^*$ and $I_n^*$ for the Erd\H{o}s--R\'enyi network ($p=0.4$). Under the limiting normal approximation, the theoretical rejection probabilities are $\alpha=0.01$ and $\alpha=0.05$ for the lower ($L$) and upper ($U$) tails, respectively.}} 
    \label{tab:erdos_probabilities} 
    \begin{tabular}{ccccccccc} 
        \hline 
        & \multicolumn{4}{c}{$r_n$} & \multicolumn{4}{c}{$I_n$} \\ \cline{2-9} & \multicolumn{2}{c}{$\alpha=0.01$} & \multicolumn{2}{c}{$\alpha=0.05$} & \multicolumn{2}{c}{$\alpha=0.01$} & \multicolumn{2}{c}{$\alpha=0.05$} \\ \cline{2-9} $n$ & $L$ & $U$ & $L$ & $U$ & $L$ & $U$ & $L$ & $U$ \\ 
        \hline 
        25  & 0.011 & 0.007 & 0.054 & 0.040 & 0.011 & 0.008 & 0.056 & 0.040 \\
        50  & 0.010 & 0.011 & 0.048 & 0.051 & 0.010 & 0.011 & 0.049 & 0.049 \\
        75  & 0.010 & 0.010 & 0.049 & 0.049 & 0.010 & 0.010 & 0.048 & 0.049 \\
        100 & 0.011 & 0.009 & 0.050 & 0.048 & 0.010 & 0.009 & 0.049 & 0.049 \\
        \hline 
    \end{tabular} 
\end{table}
\FloatBarrier

 {When $n=25$ some of the rejection probabilities in Table~\ref{tab:erdos_probabilities} differ from their limiting normal values, especially for $\alpha=0.05$. For $n\ge 50$ the conclusion of Theorem \ref{thm1} holds as the nominal and empirical rejection probabilities are equal, up to Monte Carlo errors. For all $n$,  $CV_{nd}$ is fairly small ($CV_{nd}\approx 0.15$), and $\sigma(r_n) \approx \sigma (I_n)$;  the two statistics have the same limiting normal distribution.   This network is dense ($\bar d_n/n\approx p=0.4$); Corollary~\ref{cor:sparse_regime-asym} does not apply and the ratio $\mu(I_n)/\sigma(I_n)\approx-\sqrt{p/\{2(1-p)\}} \approx-0.577$ is not negligible for large $n$.   The structural indicators of Table S3 show that the  normalized maximum eigenvalue $\lambda^*_{\text{max}}$, defined in equation \eqref{eq:spectral-cond}, converges to   0 as $n$ becomes large, for the two statistics.  The hypotheses of Theorem \ref{thm1} appear to hold for simulated Erd\H{o}s--R\'enyi networks.}\\

{Table~S1 in the Supplementary Material compares the empirical and nominal rejection probabilities for the bounded degree network. There is a close agreement between the two sets of probabilities for $n=25,\ 50,\ 75,\ 100$; Theorem \ref{thm1} appears to hold for small values of $n$.  In Table~S4 the ratios  $\mu(I_n)/\sigma(I_n)$ and $\mu(r_n)/\sigma(r_n)$ are not negligible. A sample size of $n=100$ is not large enough for Corollary~\ref{cor:sparse_regime-asym} to apply to these bounded degree networks.} \\

{For the dense heterogeneous network, Table~S2 shows important differences, especially for $r_n$, between the empirical and the nominal rejection probabilities, even at $n=100$.  The normalized maximum eigenvalues $\lambda^*_{\text{max}}$ reported in Table~S5 do not appear to converge to 0 as $n$ increases. Theorem \ref{thm1} does not apply to these networks. The normal Q--Q plots reported in the supplementary material (Figures~S1--S3) are consistent with these results.}

\subsection{Real networks}
\label{sec:real}

\subsubsection{Chesapeake and Karate networks}
{We now turn to two real networks, presented in Figure \ref{fig:real_networks}. They are used as benchmarks in the network literature, to assess how well the asymptotic normality established in Theorem~\ref{thm1} performs outside a simulated setting.  They are }
\begin{itemize}
    \item[(i)] The \textbf{Chesapeake Bay food web} (\cite{baird1989seasonal}), consisting of $n=33$ organisms and predation links. 
    \item[(ii)] The \textbf{Zachary's Karate Club social network} (\cite{zachary1977information}), which comprises $n=34$ individuals.
\end{itemize}

\begin{figure}[!h]
    \centering
    \begin{subfigure}[t]{0.4\textwidth}
        \centering
        \includegraphics[width=\textwidth]{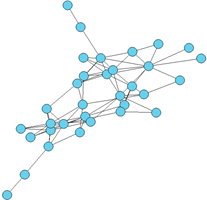}
        \caption{Chesapeake network}
    \end{subfigure}
    \hfill
    \begin{subfigure}[t]{0.4\textwidth}
        \centering
        \includegraphics[width=\textwidth]{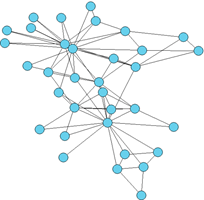}
        \caption{Karate network}
    \end{subfigure}
    \caption{Representation of the Chesapeake and Karate networks.}
    \label{fig:real_networks}
\end{figure}
\FloatBarrier

{Table~\ref{tab:empirical_quantiles} reports the empirical rejection probabilities of the one-sided tests based on $M=10\,000$ replications of the standardized statistics $r_n^*$ and $I_n^*$ for both networks. 
At the $5\%$ level the probabilities are close to the nominal level for Chesapeake  and mildly conservative for Karate. At the $1\%$ level the departures between empirical and nominal probabilities are important. This could be expected considered the $n=33$ and $n=34$ of these two networks.}

\begin{table}[!h]
\centering
\setlength{\tabcolsep}{10pt}
\caption{{Empirical rejection probabilities of the one-sided tests based on the standardized statistics $r^*_n$ and $I_n^*$ for the Chesapeake and Karate networks. Under the standard normal approximation, the theoretical rejection probabilities are $\alpha=0.01$ and $\alpha=0.05$ for the lower ($L$) and upper ($U$) tails, respectively.}}
\label{tab:empirical_quantiles}
\begin{tabular}{lcccccccc}
\hline
 & \multicolumn{4}{c}{$r_n$} & \multicolumn{4}{c}{$I_n$} \\
\cline{2-9}
 & \multicolumn{2}{c}{$\alpha=0.01$} & \multicolumn{2}{c}{$\alpha=0.05$} & \multicolumn{2}{c}{$\alpha=0.01$} & \multicolumn{2}{c}{$\alpha=0.05$} \\
\cline{2-9}
Network & $L$ & $U$ & $L$ & $U$ & $L$ & $U$ & $L$ & $U$ \\
\hline
Chesapeake & 0.012 & 0.013 & 0.051 & 0.054 & 0.014 & 0.013 & 0.051 & 0.052 \\
Karate & 0.008 & 0.012 & 0.042 & 0.047 & 0.009 & 0.014 & 0.047 & 0.049\\
\hline
\end{tabular}
\end{table}

\FloatBarrier
{The structural indicators of Table~S6 suggests that the conclusions of Theorem \ref{thm1} do not hold for these two networks as their  normalized maximum eigenvalues $\lambda^*_{\text{max}}$ are not negligible. This could be anticipated for the Karate network since, considering Figure \ref{fig:real_networks},  this network resembles the Community network of Section \ref{sec:com} with  $k=2$,  for which the conclusion of Theorem \ref{thm1} fails.}

\subsubsection{Insurance Company networks}

This section considers confidential data about clients of an insurance company. Three client networks are available together with a continuous dependent variable $Y$ measuring the involvement of a client with the company. Each network has more than 10\,000 clients. {These networks are sparse as their average degrees given in Table~S7 satisfy $\bar d_n\in[1.08,1.30]$.  Corollary~\ref{cor:sparse_regime-asym} applies:  $I_n$ and $r_n$ are, for all practical purposes, equal and normal test statistics are obtained by multiplying $r_n$ and $I_n$ by $\sqrt{n \bar d_n/2}$}.  The results are reported in Table~\ref{tab:insurance_results}:  The association with $\mathbf{Y}_n$ is slightly larger for Networks 2 and 3. Indeed, the association with Network 2 is significant at the 5\% level, because this network is larger than the third one.
Since $\bar d_n=1.22$ for Network 2, it consists mostly of disjoint dyads.  Its structure resembles that of the Community network of Section~\ref{sec:com} with $\ell_n=2$. Thus, $I_n= r_n\approx1\%$ means that this approximate dyad structure explains about 1\% of the variability in $Y$.

\begin{table}[!h]
\centering
\setlength{\tabcolsep}{6pt}
\caption{Test results for the company networks.}
\label{tab:insurance_results}
\begin{tabular}{ccccc}
\hline
Network & $r_n$ & $I_n$ & \textit{p-value}$(r_n)$ & \textit{p-value}$(I_n)$ \\
\hline
1 & 0.0069 & 0.0069 & 0.1586 & 0.1583 \\
2 & 0.0101 & 0.0102 & 0.0253 & 0.0249 \\
3 & 0.0116 & 0.0116 & 0.1675 & 0.1682 \\
\hline
\end{tabular}
\end{table}

    \section{The proofs of Theorem \ref{thm1} and Corollary \ref{cor:sparse_regime-asym}}
\label{sec:proof}
Without loss of generality we can assume that the entries of  $\mathbf{Y}_n$  are independent $N(0,1)$ variables since $r_n$ and $I_n$  are invariant with respect to a rescaling of the variables.

\subsection{Evaluation of the first two moments of $S_{A,n}$ and $S_{AM,n}$ 
defined in \eqref{eq:ratio}}
The first step in the proof is to calculate the expectation and the variance of quadratic forms in normal variables.  To do so we use standard results  \cite{ghazal1996}  that states that if $\mathbf{Y}_n \sim N_n(0,\mathbf{I}_n)$ then for any symmetric matrix $A$, $\mathbb{E}(\mathbf{Y}_n^\top A \mathbf{Y}_n)=\mbox{tr}(A)$ while $\mathrm{Var}(\mathbf{Y}_n^\top A \mathbf{Y}_n)=2\mbox{tr}(A^2)$. The first two moments of the numerators in expressions \eqref{eq:ratio}, obtained with these formulas, are given in the next lemma.
	\begin{lemma}\label{lem:moments-asym}
		Let $\mu(r_n),\ \mu(I_n), \ \sigma^2(r_n)$ and $\sigma^2(r_n)$ be the asymptotic moments defined in Theorem \ref{thm1} and $\mathbf{Y}_{n}$ be a random vector with a $N_n(0,\mathbf{I}_n)$ distribution, then 
		\begin{eqnarray*}
			\mathbb{E}[ S_{A,n}] & = & \mu(r_n)=-\frac 1 {n}\,\bigl(CV_{nd}^{2}+1\bigr),   \\
            \mathbb{E}[ S_{AM,n}] & = & \mu(I_n)=-\frac{1}{n},\\  
            \mathrm{Var}( S_{A,n})& = & \sigma^2(r_n)= \frac{2}{n^{2}}
			\left[
			\frac{n}{\bar d_{n}}+ CV_{nd}^{4}-1-2CV_{nd}^{2}I_{nd}-2CV_{nd}^{2}
			\right],\\
			  \mathrm{Var}( S_{AM,n}) & = & \sigma^2(I_n)=\frac{2}{n^{2}}
			\left[
			\frac{n}{\bar d_{n}}-1-2CV_{nd}^{2}
			\right],
		\end{eqnarray*}
	\end{lemma}
\begin{proof}
For an adjacency matrix $A_n$, we have 
$\operatorname{tr}(A_n)=0$ and  $ \operatorname{tr}(A_n^2)=2m_n=\sum d_{ni}=n\bar d_n$. 
Thus one has
	\begin{eqnarray*}
			\operatorname{E}\{ S_{A,n} \}
			& = &\frac{1}{2m_{n}} \operatorname{tr}
			\!\left[ \left( A_{n} - \frac{1}{2m_{n}}\mathbf{d}_{n}\mathbf{d}_{n}^{\top}\right)
			\mathbf{I}_{n} \right]  \\
			& = & -\frac{1}{4m_n^2} \sum_{i=1}^{n} d_i^2 = -\frac{1}{n}(CV_{nd}^{2}+1).\\
			\operatorname{var}\{ S_{A,n} \}
			& = & \frac{2}{4m_n^{2}}
			\operatorname{tr}\!\left[\left( A_{n} - \frac{1}{2m_{n}}\mathbf{d}_{n}\mathbf{d}_{n}^{\top} \right)^{2}\right] \\
			& = &\frac{2}{4m_{n}^{2}}\left[2m_{n}  - \frac{2}{2m_{n}}\mathbf{d}_{n}^{\top}A_{n}\mathbf{d}_{n}+\frac{\left(\sum d_i^2\right)^2}{4m_n^2}\right]\\
            & = & \frac{ 2}{n^{2}} \left[ \frac{n}{\bar d_{n}}+
			CV_{nd}^{4} - 2CV_{nd}^{2}I_{nd}- 1 -2CV_{nd}^{2}  \right]. 
       	\end{eqnarray*}   
The last expression is obtained with the following formula for Moran's $I$ calculated with  the degree vector
$$
I_{nd}= \frac{\mathbf{d}_{n}^{\top}A_{n}\mathbf{d}_{n}}{s_{nd}^{2}\  2m_{n}} -2 -\frac{\bar d_n^{2}}{s_{nd}^{2}},
$$
so that $\mathbf{d}_{n}^{\top}A_{n}\mathbf{d}_{n}/(2m_n)=s_{nd}^2I_{nd} +2s_{nd}^2 +\bar d_n^2$,  and the fact that $ \left(\sum d_i^2\right)^2/(4m_n^2)=\bar d_n^2(CV_{nd}^4+2CV_{nd}^2+1)$.
For the numerator of Moran's $I$ statistic one has
        \begin{eqnarray*}
			\operatorname{E}\{ S_{AM,n} \}
			& = & \frac{1}{2m_{n}} \operatorname{tr}
			\!\left[ M_nA_{n}M_n \right]=\frac{1}{2m_{n}} \operatorname{tr}
			\!\left[ A_{n}M_n \right]=  -\frac{1}{n}\\
			\operatorname{var}\{ S_{AM,n} \}
			& = & \frac{2}{4m_n^{2}}
			\operatorname{tr}\!\left[\left(M_n A_{n} M_{n} \right)^{2}\right]=\operatorname{tr}\!\left[\left(M_n A_{n}  \right)^{2}\right] \\
			& = & \frac{2}{4m_{n}^{2}}\left[2m_{n}-\frac{2}{n} \overunderset{n}{i=1}{\sum}d_{ni}^{2} +\frac{4m_n^2}{n^2}\right]\\
			& = &\frac{ 2}{n^{2}} \left[ \frac{n}{\bar d_{n}}-2CV_{nd}^{2} - 1 \right].
	\end{eqnarray*}
\end{proof}
\subsection{Proof of Theorem \ref{thm1}} 

We first show that the Lindeberg condition \eqref{condLind} is verified for the numerators, $S_{A,n}$ and $S_{AM,n}$, of $r_n$ and $I_n$. Both are quadratic forms in normal variables constructed with different matrices $B_n$.  These quadratic forms are distributed as a linear combination of $\chi_1^2$ random variables,
        $$
        S_{a,n} \stackrel{D}{=} \frac{1}{2m_{n}}\sum_{i=1}^{n}\lambda_{ni}Z_{i}^{2},
        $$
		where $\lambda_{n1},\ldots,\lambda_{nn}$ are the eigenvalues of the matrix $B_n$, and the $Z_i$'s are independent $\mathcal{N}(0,1)$ random variables.  We want to apply Theorem \ref{thm3} to the triangular array of the variables $X_{n,i}=\lambda_{ni}Z_i^2/(2m_n)$.  
		We have
		\[
		\sum_{i=1}^{n}\mu_{ni} =\sum_{i=1}^{n}\operatorname{E}(X_{n,i})
		= \frac{1}{2m_{n}}\sum_{i=1}^{n}\lambda_{ni}
		= \frac{1}{2m_{n}}\operatorname{tr}(B_n),
		\]
		while
		\[
		s_{n}^{2}= \sum_{i=1}^{n} \mathrm{Var}(X_{n,i})
		= \frac{2}{4m_{n}^{2}}\sum_{i=1}^{n}\lambda_{ni}^{2}
        = \frac{2}{4m_{n}^{2}}\operatorname{tr}(B_n^2).
		\]
    To verify Lindeberg's condition, we need to show that \eqref{condLind} is satisfied.
		For any $\epsilon>0$, consider
		\[
		L_{n}(\epsilon)
		= \frac{1}{s_{n}^{2}}\sum_{i=1}^{n}
		E\!\left\{(X_{ni}-\mu_{ni})^{2}
		\mathbf{1}_{\{|X_{ni}-\mu_{ni}|>\epsilon s_{n}\}}\right\}.
		\]
		The expectations in the sum can be expressed as integrals with respect to the $\chi_1^2$ density:
		\[
		L_{n}(\epsilon)= \sum_{i=1}^n\frac{\lambda_{ni}^{2}}{s_n^2  4m_{n}^{2}}
		\int_0^{\infty} (t-1)^{2}
		\frac{t^{-1/2}}{\sqrt{2\pi}} e^{-t/2}
		\mathbf{1}_{\left\{
			|t-1|>\frac{2m_{n}s_{n}}{|\lambda_{ni}|}\epsilon
			\right\}} 
		dt.
		\] 
		Bounding $|\lambda_{ni}|$ by $\lambda_{n\max}$ in the integrals, we obtain
		\[
		L_{n}(\epsilon)
		\le
		\int (t-1)^{2}
		\frac{t^{-1/2}}{\sqrt{2\pi}} e^{-t/2}
		\mathbf{1}_{\left\{
			|t-1|>\frac{2m_{n}s_{n}}{\lambda_{n\max}}\epsilon
			\right\}} dt.
		\]
		Under assumption \eqref{eq:spectral-cond},
		$2m_{n}s_{n}/\lambda_{n\max}$ tends to $ +\infty$
        and $L_{n}(\epsilon)$ converges to 0. Thus, if condition \eqref{eq:spectral-cond} is verified for the
        matrix $B_n$ of a quadratic form, { 
        then Theorem \ref{thm3} yields 
        \[
        \frac{S_{AM,n}-\mu(I_n)}{\sigma(I_n)}
        \xrightarrow[n\to\infty]{\mathcal{L}}
        \mathcal N(0,1) \text{~~and~~}
         \frac{S_{A,n}-\mu(r_n)}{\sigma(r_n)}
        \xrightarrow[n\to\infty]{\mathcal{L}}
        \mathcal N(0,1).
        \]
        To complete the derivation of the asymptotic distribution of $I_n$ observe that
        \[
        \frac{I_n-\mu(I_n)}{\sigma(I_n)}
        =
        \frac{S_{AM,n}-\mu(I_n)}{\sigma(I_n)}
        \frac{1}{S_{M,n}}
        +
        \frac{\mu(I_n)}{\sigma(I_n)}
        \left(\frac{1}{S_{M,n}}-1\right) .
        \]
        Since $\mu(I_n)$ is $O(1/n)$ and $n\sigma(I_n)$ is assumed to be bounded away from 0, $S_{AM,n}$ converges to 1 in probability and the second term of the right hand side of the above expression converges to 0 in probability while the first term is asymptotically equivalent to $\left\{S_{AM,n}-\mu(I_n)\right\}/\sigma(I_n)$.  This completes the proof of the asymptotic normality of $I_n$. 
        The proof for $r_n$ is similar, since $n\sigma(r_n)$  bounded away from 0 for large $n$ implies that  Lemma \ref{lem:denom} applies and that $S_{AM,n}$ converges to 1 in probability. 
        } 
	\subsection{Proof of Corollary \ref{cor:sparse_regime-asym}}\label{sec:cor1}
When $CV_{nd}$ is bounded and $\bar d_n \mbox{~is~} o(n)$,  both $\sigma^2(I_n)$ and $\sigma^2(I_n)$ are equal to $2/(n\bar d_n)$ up to negligible terms and the ratios $\mu(r_n)/\sigma(r_n)$ and $\mu(I_n)/\sigma(I_n)$ converge to 0.  In addition, the sum of the squared eigenvalues of the matrices $B_n$ for the two statistics satisfy $\operatorname{tr}(B_n^2)/(2m_n)$ is $O(1)$, where $2m_n=n\bar d_n=\operatorname{tr}(A_n^2)$.\\

To establish that \eqref{eq:spectral-cond} holds we need to show that, for the two matrices $B_n$, $\lambda_{n,\max}/(2m_n)$ tends to 0 as $n$ goes to $\infty$. We first prove that this is true for the adjacency matrix $A_n$. Let $\lambda_1(\cdot)\le ...\le \lambda_n(\cdot)_n$ denote the ordered eigenvalue of a matrix.  Since the Laplacian $D_n-A_n$ and the matrix $D_n+A_n$ are semi definite positive \cite[Section 6.14.4]{newman2018networks}, one has 
\begin{equation} \label{eq:lmax}
\max\{-\lambda_1(A_n),\lambda_n(A_n)\} \le d_{n,\max}
\end{equation}
and condition \eqref{eq:spectral-cond} is verified for $A_n$. 
To prove that it is verified for the $B_n$ matrices for $r_n$ and $I_n$, we use the 
 following result on page 269 of \cite{golub2013matrix},
	\begin{corollary}\label{cor:8.1-3}
		If $A$ and $A+E$ are two symmetric matrices of size $n\times n$, then for any
		$k = 1,2,\ldots,n$,
		\begin{equation}\label{eq:Wilk-bound-fr}
			\lambda_k(A) + \lambda_1(E)
			\;\le\; \lambda_k(A+E)
			\;\le\; \lambda_k(A) + \lambda_n(E).
		\end{equation}
	\end{corollary}
First, we apply Corollary \ref{cor:8.1-3} to the matrix $B_n$ for $r_n$.  
We write $B_n = A_n + E_n$, where  
\[
E_n = - \frac{\mathbf{d}_n \mathbf{d}_n^{\top}}{2m_n},
\]
and $E_n$ has a single non-null eigenvalue given by  
\[
\lambda_1(E_n)
  = -\frac{\sum_i d_{ni}^2}{2m_n}
  = - \bar d_n \left( CV_{nd}^2 + 1 \right).
\]

Thus, using the bound in \eqref{eq:lmax}, we obtain
\[
\lambda_{n,\max}
  \le \max \Big\{ -\lambda_1(A_n) + \bar d_n (CV_{nd}^2 +1),
                  \lambda_n(A_n) \Big\}
  \le d_{n,\max} + \bar d_n (CV_{nd}^2 + 1).
\]
This implies that $B_n$ satisfies \eqref{eq:spectral-cond} 
as $\bar d_n (CV_{nd}^2 +1)/\sqrt{n \bar d_n}$ converges to 0
as $n$ goes to $\infty$.\\
Now we consider the matrix $B_n=M_nA_nM_n$ for $I_n$.  It has 
\begin{eqnarray*}
    E_n & = &-\frac 1 n \mathbf{1}_n \mathbf{d}_n^\top -
\frac 1 n  \mathbf{d}_n\mathbf{1}_n ^\top +
\frac 1 n  \mathbf{1}_n\mathbf{1}_n ^\top\\
& = & \left[\frac{\mathbf{1}_n}{\sqrt{n}};\frac{\mathbf{d}_n-\bar d_n \mathbf{1}_n}{\sqrt{ns_d^2}} \right]
\begin{pmatrix}
    1-2\bar d_n & -\sqrt{s_d^2} \\-\sqrt{s_d^2} & 0 
\end{pmatrix}
\left[\frac{\mathbf{1}_n}{\sqrt{n}};\frac{\mathbf{d}_n-\bar d_n \mathbf{1}_n}{\sqrt{ns_d^2}} \right]^\top.
\end{eqnarray*}
This matrix has two non-null eigenvalues equal to the those of the middle $2 \times 2$ matrix in the expression above.  They are equal to $(1-2\bar d_n \pm \sqrt{(1-2\bar d_n)^2+4s_d^2})/2$ and they converge to 0 when divided by $\sqrt{ n \bar d_n}$.  Thus the proof that $\lambda_{n,\max}$, for statistic $I_n$, fulfills condition \eqref{eq:spectral-cond} is similar to that for $r_n$.

\section*{Conclusion}
	
	In this article, we conducted a detailed investigation of the asymptotic behavior of two network-based dependence statistics: the assortativity coefficient $r_n$ and Moran’s coefficient $I_n$. By adopting a unified framework based on quadratic forms, we established simple yet general conditions that guarantee their asymptotic normality under the assumptions
	of normality and independence of the attribute values.
    The results of the paper can be extended in several directions.  First, Theorem \ref{thm1} can be applied to an arbitrary symmetric proximity matrix $A_n$ with non negative entries.  Second, the asymptotic distribution of Moran's $I$ calculated on regression residuals can be deduced from Theorem \ref{thm1}.  The numerator of Moran's $I$ is then proportional to  $\mathbf{Y}_n^\top M_n^*A_nM_n^*\mathbf{Y}_n $ where $M_n^*$ is an idempotent matrix defined in such a way that $M_n^*\mathbf{Y}_n$ is the vector of regression residuals. To derive the asymptotic distribution of $I_n$, one applies Theorem \ref{thm1} to a quadratic form defined by the matrix $B_n=M_n^*A_nM_n^*$ with $\mu(I_n)=\operatorname{tr}(B_n)/(2m_n)$ and 
    $\sigma^2(I_n)=2\operatorname{tr}(B_n^2)/(4m_n^2)$.\\
    
    {Third, the Gaussian assumption on $Y$ can be relaxed.  Using the argument in  Sen, \cite{sen1976large}, we can generalize  Corollary \ref{cor:sparse_regime-asym} to random variables $Y$ satisfying $E(Y^4)<\infty$, provided that the degrees are uniformly bounded by a fixed constant (these are ``finitely contiguous'' weights, in Sen's terminology). \cite{kelejian_prucha2001} further extend Sen's argument to independent, possibly heteroskedastic and non-normal innovations, and to the generalized Cliff–Ord specification. Extending it to networks satisfying only the weaker sparsity condition of Corollary \ref{cor:sparse_regime-asym} where the degree may grow with $n$, remains open and would require adapting Sen's martingale argument to our spectral condition (\ref{eq:spectral-cond}).}\\

    
    {Fourth, our results are stated conditionally on an observed adjacency matrix. Corollary \ref{cor:random_regular_dense-asym} is a first step beyond this setting, since it verifies condition \eqref{eq:spectral-cond} with probability at least $1-1/n$ when $A_n$ is drawn uniformly among regular graphs of degree $d_n\approx\alpha n$. Establishing analogous results for the $\beta$-model, the $p_0$ model or covariate-adjusted network models would be a natural continuation.
    }
    
    \section*{Acknowledgments}
We also acknowledge the industrial partner Prompt as well as NSERC for their financial support of the research project within which this article was written.\\

\bibliographystyle{comnet}
\bibliography{ref_jcn}

\newpage

\renewcommand{\thetable}{S\arabic{table}}
\renewcommand{\thefigure}{S\arabic{figure}}

\centerline{\bf Supplementary materials}

\setcounter{section}{0}
\setcounter{table}{0}
\setcounter{figure}{0}

\medskip

\noindent 
This supplementary document complements the main manuscript. References to theorems, corollaries and equations refer to the numbering used in the main article.

\section{Comparison of the empirical and theoretical probabilities}

This section reports the empirical rejection probabilities not already shown in the main text. Those for the Erd\H{o}s--R\'enyi, Chesapeake and Karate networks appear respectively in Tables~1 and 2 of the main article; the insurance company networks are not included here, as individual-level attribute data are unavailable for computing standardized statistics.

\begin{table}[!h] 
\centering 
\setlength{\tabcolsep}{10pt} 
\caption{Empirical rejection probabilities of the one-sided tests based on the standardized statistics for bounded-degree networks. Under the standard normal approximation, the theoretical rejection probabilities are $\alpha=0.01$ and $\alpha=0.05$ for the lower ($L$) and upper ($U$) tails, respectively.} 
\label{tab:bounded_probabilities} 
\begin{tabular}{ccccccccc} 
\hline 
& \multicolumn{4}{c}{$r_n$} & \multicolumn{4}{c}{$I_n$} \\ \cline{2-9}
& \multicolumn{2}{c}{$\alpha=0.01$} & \multicolumn{2}{c}{$\alpha=0.05$} & \multicolumn{2}{c}{$\alpha=0.01$} & \multicolumn{2}{c}{$\alpha=0.05$} \\ \cline{2-9}
$n$ & $L$ & $U$ & $L$ & $U$ & $L$ & $U$ & $L$ & $U$ \\ 
\hline 
25  & 0.010 & 0.010 & 0.046 & 0.046 & 0.010 & 0.008 & 0.051 & 0.045 \\
50  & 0.011 & 0.010 & 0.052 & 0.048 & 0.011 & 0.009 & 0.054 & 0.047 \\
75  & 0.011 & 0.010 & 0.051 & 0.053 & 0.011 & 0.010 & 0.051 & 0.052 \\
100 & 0.009 & 0.010 & 0.053 & 0.050 & 0.010 & 0.010 & 0.053 & 0.049 \\
\hline 
\end{tabular} 
\end{table}
\begin{table}[!h]
    \centering 
    \setlength{\tabcolsep}{10pt} 
    \caption{Empirical rejection probabilities of the one-sided tests based on the standardized statistics for dense heterogeneous networks. Under the standard normal approximation, the theoretical rejection probabilities are $\alpha=0.01$ and $\alpha=0.05$ for the lower ($L$) and upper ($U$) tails, respectively.}
    \label{tab:dense_probabilities} 
    \begin{tabular}{ccccccccc} 
        \hline & \multicolumn{4}{c}{$r_n$} & \multicolumn{4}{c}{$I_n$} \\ 
        \cline{2-9} & \multicolumn{2}{c}{$\alpha=0.01$} & \multicolumn{2}{c}{$\alpha=0.05$} & \multicolumn{2}{c}{$\alpha=0.01$} & \multicolumn{2}{c}{$\alpha=0.05$} \\
        \cline{2-9} $n$ & $L$ & $U$ & $L$ & $U$ & $L$ & $U$ & $L$ & $U$ \\
        \hline
        25  & 0.009 & 0.010 & 0.047 & 0.045 & 0.008 & 0.007 & 0.048 & 0.044 \\
        50  & 0.012 & 0.008 & 0.057 & 0.042 & 0.010 & 0.009 & 0.051 & 0.045 \\ 
        75  & 0.015 & 0.006 & 0.052 & 0.038 & 0.013 & 0.008 & 0.050 & 0.046 \\ 
        100 & 0.015 & 0.006 & 0.049 & 0.040 & 0.009 & 0.010 & 0.044 & 0.046 \\
        \hline 
    \end{tabular} 
\end{table}
\FloatBarrier
\section{Structural indicators}

For each network, we computed the following structural indicators:
\[
\bar d_n,
\qquad
CV_{nd}, \qquad
I_{nd},
\qquad
\lambda_{\max}^*=\frac{\lambda_{n\max}}{\sqrt{\sum_{i=1}^{n}\lambda_{ni}^{2}}},
\qquad
\alpha_{\max}^*=\frac{\alpha_{n\max}}{\sqrt{\sum_{i=1}^{n}\alpha_{ni}^{2}}},
\]
where $(\lambda_{ni})$ and $(\alpha_{ni})$ denote, respectively, the eigenvalues of the matrices
\[
A_n - \frac{1}{2m_n}A_n\mathbf 1_n \mathbf 1_n^\top A_n,
\qquad
M_n A_n M_n,
\]
which appear in the definitions of the two statistics. For the simulated networks, we generated $10\,000$ independent realizations of the random vector $\mathbf{Y}_n$, and for each realization computed the standardized test statistics
\[
r_n^\ast=\frac{r_{n}-\mu(r_{n})}{\sigma(r_{n})},
\qquad
I_n^\ast=\frac{I_{n}-\mu(I_{n})}{\sigma(I_{n})}.
\]

\begin{table}[!h] 
\centering  
\caption{Structural indicators for the Erd\H{o}s--R\'enyi networks ($p=0.4$).} 
\label{tab:erdos_indicators} 
\begin{tabular}{ccccccc} 
\hline
$n$ & $\bar d_n$ & $CV_{nd}$ & $I_{nd}$ & $\lambda_{\max}^{*}$ & $\alpha_{\max}^{*}$ & $\dfrac{\mu(r_n)}{\sigma(r_n)}$ / $\dfrac{\mu(I_n)}{\sigma(I_n)}$\\ 
\hline 
25  & 10.880 & 0.157 & -0.016 & 0.412 & 0.412 & -0.648 / -0.633 \\            
50  & 21.040 & 0.158 & -0.033 & 0.271 & 0.271 & -0.629 / -0.614 \\
75  & 29.733 & 0.149 & -0.007 & 0.254 & 0.254 & -0.594 / -0.582 \\ 
100 & 39.360 & 0.124 & -0.013 & 0.198 & 0.198 & -0.584 / -0.575 \\
\hline 
\end{tabular} 
\end{table}

\begin{table}[!h] 
\centering
\caption{Structural indicators for the bounded-degree networks.} 
\label{tab:bounded_indicators} 
\begin{tabular}{cccccccc} 
\hline
$n$ & $\bar d_n$ & $CV_{nd}$ & $I_{nd}$ & $\lambda_{\max}^{*}$ & $\alpha_{\max}^{*}$ & $\dfrac{\mu(r_n)}{\sigma(r_n)}$ & $\dfrac{\mu(I_n)}{\sigma(I_n)}$\\ 
\hline 
25  & 10.400 & 0.462 & -0.020 & 0.392 & 0.401 & -0.844 & -0.715 \\            
50  & 11.080 & 0.503 & 0.167  & 0.293 & 0.292 & -0.513 & -0.408 \\
75  & 9.520  & 0.552 & 0.193  & 0.254 & 0.253 & -0.369 & -0.282 \\ 
100 & 10.820 & 0.485 & 0.156  & 0.215 & 0.215 & -0.314 & -0.254 \\
\hline 
\end{tabular} 
\end{table}

\begin{table}[!h] 
\centering  
\caption{Structural indicators for the dense heterogeneous networks.} 
\label{tab:dense_indicators} 
\begin{tabular}{cccccccc} 
\hline $n$ & $\bar d_n$ & $CV_{nd}$ & $I_{nd}$ & $\lambda_{\max}^{*}$ & $\alpha_{\max}^{*}$ & $\dfrac{\mu(r_n)}{\sigma(r_n)}$ & $\dfrac{\mu(I_n)}{\sigma(I_n)}$\\ 
\hline 
25  & 12.400 & 0.412 & -0.022 & 0.428 & 0.411 & -0.979 & -0.859 \\
50  & 26.000 & 0.442 & -0.043 & 0.406 & 0.287 & -1.103 & -0.970 \\ 
75  & 37.333 & 0.510 & -0.031 & 0.519 & 0.338 & -1.178 & -1.012 \\
100 & 49.460 & 0.519 & -0.023 & 0.507 & 0.300 & -1.191 & -1.017 \\
\hline 
\end{tabular} 
\end{table}

\begin{table}[!h]
\centering
\setlength{\tabcolsep}{11pt}
\caption{Structural indicators for the Chesapeake and Karate networks.}
\label{tab:realnet_structural}
\begin{tabular}{lcccccccc}
\hline
Network & $n$ & $\bar d_n$ & $CV_{nd}$ & $I_{nd}$ & $\lambda_{\max}^*$ & $\alpha_{\max}^*$ & $\dfrac{\mu(r_n)}{\sigma(r_n)}$ & $\dfrac{\mu(I_n)}{\sigma(I_n)}$ \\
\hline
Chesapeake & 33 & 4.303 & 0.568 & 0.14 & 0.44 & 0.44 & -0.38 & -0.29 \\
Karate & 34 & 4.59 & 0.833 & -0.24 & 0.50 & 0.48 & -0.5 & -0.32 \\
\hline
\end{tabular}
\end{table}

\begin{table}[!h]
\centering
\caption{Structural indicators for the company networks.}
\label{tab:insurance_structural}
\begin{tabular}{ccc}
\hline
Network & $\bar d_n$ & $CV_{nd}$ \\
\hline
1 & 1.08 & 0.274 \\
2 & 1.22 & 0.453 \\
3 & 1.30 & 0.503 \\
\hline
\end{tabular}
\end{table}

\section{Q--Q plots}

Figures~\ref{fig:all_qqplot_ER}--\ref{fig:all_qqplot_real} display the Q--Q plots of the standardized statistics against the standard normal distribution, for $n=25,50,75,100$ (simulated networks) or for the observed network (Chesapeake and Karate). The sample quantiles lie close to the reference line over the central range in all cases, with departures confined to the tails; these are most visible for $r_n$ in the dense heterogeneous network.

\begin{figure}[!h]
\centering
\begin{subfigure}[t]{0.7\textwidth}
    \centering
    \includegraphics[width=\textwidth]{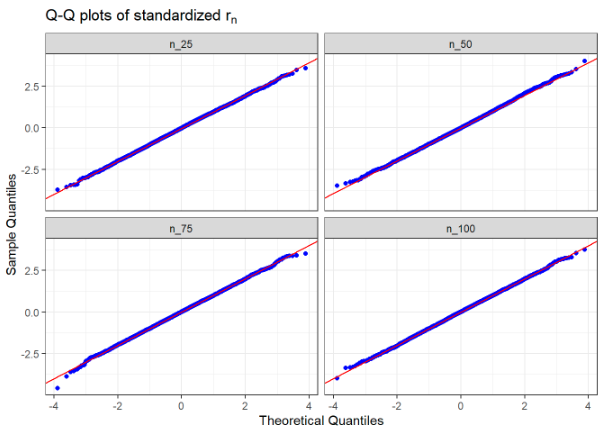}
\end{subfigure}
\hfill
\begin{subfigure}[t]{0.7\textwidth}
    \centering
    \includegraphics[width=\textwidth]{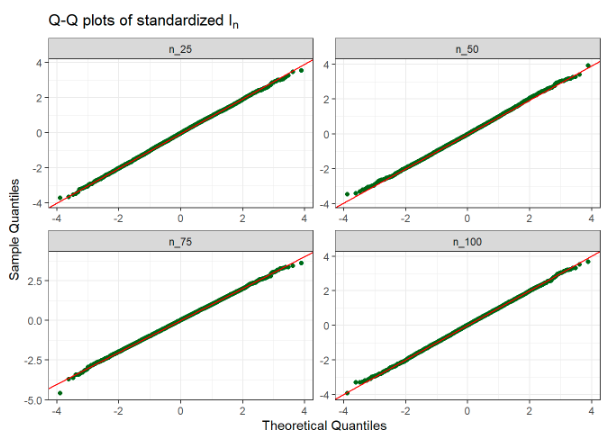}
\end{subfigure}
\caption{Normal Q--Q plots of the standardized statistics $r_n^\ast$ and $I_n^\ast$ for Erd\H{o}s--R\'enyi networks of sizes $n=25$, $50$, $75$, and $100$. The observed quantiles are compared with the theoretical quantiles of the standard normal distribution. The red line denotes the normal reference line.}
\label{fig:all_qqplot_ER}
\end{figure}

\begin{figure}[!h]
\centering
\begin{subfigure}[t]{0.7\textwidth}
    \centering
    \includegraphics[width=\textwidth]{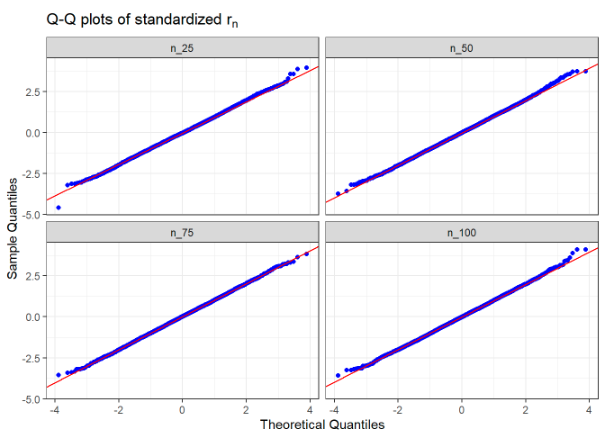}
\end{subfigure}
\hfill
\begin{subfigure}[t]{0.7\textwidth}
    \centering
    \includegraphics[width=\textwidth]{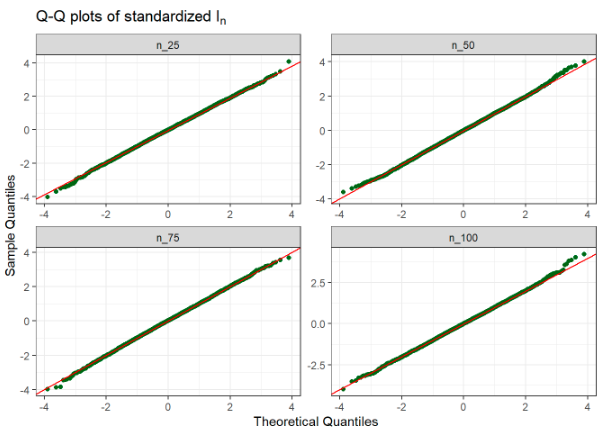}
\end{subfigure}
\caption{Normal Q--Q plots of the standardized statistics $r_n^\ast$ and $I_n^\ast$ for bounded-degree networks of sizes $n=25$, $50$, $75$, and $100$. The observed quantiles are compared with the theoretical quantiles of the standard normal distribution. The red line denotes the normal reference line.}
\label{fig:all_qqplot_bounded}
\end{figure}

\begin{figure}[!h]
\centering
\begin{subfigure}[t]{0.7\textwidth}
    \centering
    \includegraphics[width=\textwidth]{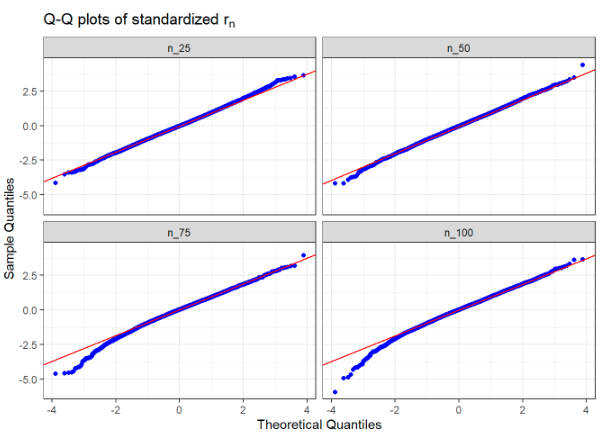}
\end{subfigure}
\hfill
\begin{subfigure}[t]{0.7\textwidth}
    \centering
    \includegraphics[width=\textwidth]{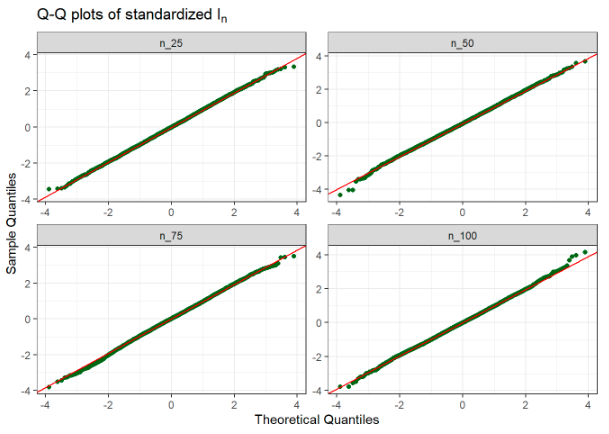}
\end{subfigure}
\caption{Normal Q--Q plots of the standardized statistics $r_n^\ast$ and $I_n^\ast$ for dense heterogeneous networks of sizes $n=25$, $50$, $75$, and $100$. The observed quantiles are compared with the theoretical quantiles of the standard normal distribution. The red line denotes the normal reference line.}
\label{fig:all_qqplot_dense}
\end{figure}

\begin{figure}[!h]
\centering
\begin{subfigure}[t]{0.7\textwidth}
    \centering
    \includegraphics[width=\textwidth]{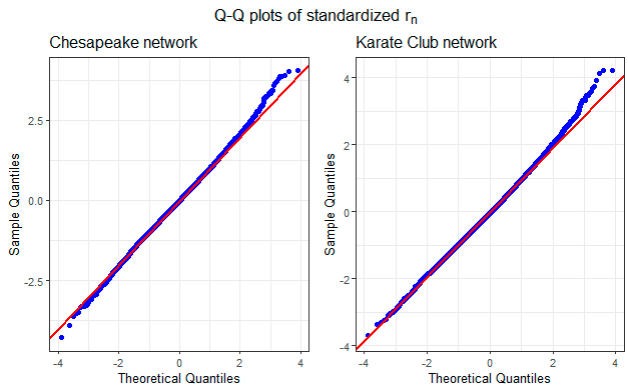}
\end{subfigure}
\hfill
\begin{subfigure}[t]{0.7\textwidth}
    \centering
    \includegraphics[width=\textwidth]{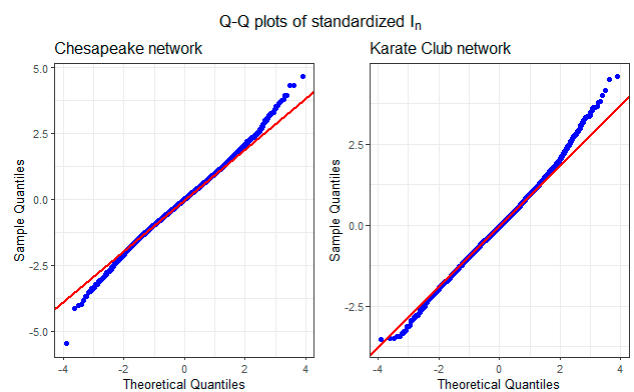}
\end{subfigure}
\caption{Normal Q--Q plots of the standardized statistics $r_n^\ast$ and $I_n^\ast$ for the Chesapeake and Karate networks. The observed quantiles are compared with the theoretical quantiles of the standard normal distribution. The red line denotes the normal reference line.}
\label{fig:all_qqplot_real}
\end{figure}

\end{document}